\documentclass[11pt]{article}
\usepackage[numbers,sort&compress]{natbib}
\usepackage{enumerate}
\usepackage{amscd}
\usepackage{amsmath}
\usepackage{latexsym}
\usepackage{amsfonts}
\usepackage{amssymb}
\usepackage{amsthm}
\usepackage{verbatim}
\usepackage{mathrsfs}
\usepackage{enumerate}
\usepackage{hyperref}

\theoremstyle{plain}\newtheorem{definition}{Definition}[section]
\theoremstyle{definition}\newtheorem{theorem}{Theorem}[section]
\theoremstyle{plain}\newtheorem{lemma}[theorem]{Lemma}
\theoremstyle{plain}\newtheorem{coro}[theorem]{Corollary}
\theoremstyle{plain}
\theoremstyle{remark}\newtheorem{remark}{Remark}[section]
\usepackage{xcolor}

\newcommand{\norm}[1]{\left\|#1\right\|}
\newcommand{\Div}{\mathrm{div}\,}
\newcommand{\B}{\Big}

\newcommand{\be}{\begin{equation}}
\newcommand{\ee}{\end{equation}}
 \newcommand{\ba}{\begin{aligned}}
 \newcommand{\ea}{\end{aligned}}

  \newcommand{\f}{\frac}
    
  \newcommand{\ben}{\begin{enumerate}}
   \newcommand{\een}{\end{enumerate}}

\newcommand{\ti}{\nabla}

\newcommand{\Rmnum}[1]{\expandafter\@slowromancap\romannumeral #1@}

\allowdisplaybreaks

\numberwithin{equation}{section}
\begin{document}
\title{  Energy and helicity   conservation for the generalized quasi-geostrophic equation }
\author{Yanqing Wang\footnote{   College of Mathematics and   Information Science, Zhengzhou University of Light Industry, Zhengzhou, Henan  450002,  P. R. China Email: wangyanqing20056@gmail.com},  ~   \,   Yulin Ye\footnote{School of Mathematics and Statistics,
Henan University,
Kaifeng, 475004,
P. R. China. Email: ylye@vip.henu.edu.cn} ~   \, and \, Huan Yu \footnote{ School of Applied Science, Beijing Information Science and Technology University, Beijing, 100192, P. R. China Email:  yuhuandreamer@163.com}
 }
\date{}
\maketitle
\begin{abstract}
In this paper, we consider the 2-D  generalized  surface  quasi-geostrophic  equation with the velocity $v$ determined by $v=\mathcal{R}^{\perp}\Lambda^{\gamma-1}\theta$. It is shown that
the $L^p$ type  energy norm  of weak solutions  is conserved provided $\theta\in L^{p+1}(0,T; {B}^{\f{\gamma}{3}}_{p+1, c(\mathbb{N})})$ for $0<\gamma<\f32$ or $\theta\in  L^{p+1}(0,T; {{B}}^{\alpha}_{p+1,\infty})~\text{for any}~\gamma-1<\alpha<1 \text{ with} ~\frac{3}{2}\leq \gamma <2$. Moreover, we also prove that the helicity  of weak solutions satisfying  $\nabla\theta \in L^{3}(0,T;\dot{B}_{3,c(\mathbb{N})}^{\f{\gamma}{3}})$ for $0<\gamma<\f32$ or $\nabla\theta\in  L^{3}(0,T; \dot{B}^{\alpha}_{3,\infty})~\text{for any}~\gamma-1<\alpha<1 \text{ with} ~\frac{3}{2}\leq \gamma <2$ is   invariant. Therefore, the  accurate relationships between the critical regularity for the energy (helicity) conservation of the weak solutions and the regularity of  velocity  in 2-D generalized  quasi-geostrophic equation are presented.
  \end{abstract}
\section{Introduction}
\label{intro}
\setcounter{section}{1}\setcounter{equation}{0}
In this paper, we consider   the  2-D   generalized surface quasi-geostrophic (SQG)   equation $\text{in } ~(0,T)\times \mathbb{R}^2$ below
\be\left\{\ba\label{gqg}
&\theta_{t} + v\cdot\ti
\theta=0,\\
&v=\mathcal{R}^{\perp}\Lambda^{\gamma-1}\theta=
(-\mathcal{R}_{2}\Lambda^{\gamma-1}\theta,\mathcal{R}_{1}\Lambda^{\gamma-1}\theta),~~\gamma\in[0,2],\\
&\theta|_{t=0}=\theta_0,
\ea\right.\ee
where the unknown scalar function
$\theta(x, t)\colon \mathbb{R}^2\to \mathbb{R}$ stands for  the temperature  and $v$ is the velocity field.                      The Riesz transforms $\mathcal{R}_{j}$ are defined by $\widehat{\mathcal{R}_{j}f}=-\f{i\xi_{j}}{|\xi|}\widehat{f}(\xi)$ with $j=1,2$, where $\hat{f}(\xi)=\frac{1}{(2\pi)^{2}}\int_{\mathbb{R}^{2}}f (x)e^{-i\xi\cdot x}\,dx$.
  $\Lambda^{s}f $ is defined via
$\widehat{\Lambda^{s} f}(\xi)=|\xi|^{s}\hat{f}(\xi).$
 This model was introduced in \cite{[CCW],[CCW1],[CCCGW],[CIW]} and includes many classical hydrodynamic equations. In
particular, there hold
 	\begin{enumerate}
 		\item  \eqref{gqg}  with   $\gamma=0$  reduces to  the 2-D
vorticity equation of the incompressible Euler equations \eqref{Euler} below,
  \begin{equation}\left\{\begin{aligned}\label{Euler}
&v_{t}+v\cdot \nabla v+\nabla \Pi=0,\\ &\text{div}\, v=0,\\
&v|_{t=0}=v_{0}(x).
\end{aligned}\right.\end{equation}
	\item \eqref{gqg}  with
$\gamma=1$  is the following standard surface  quasi-geostrophic equation  \cite{[CMT]};
\be\left\{\ba\label{qg1}
&\theta_{t} +\Div( v\otimes
\theta)=0,\\
&v(x,t)=(-\mathcal{R}_2 \theta, \mathcal{R}_1 \theta ),\\
&\theta|_{t=0}=\theta_0,
\ea\right.\ee
		\item \eqref{gqg}  with  $\gamma=2$ becomes the  magneto-geostrophic equations \cite{[FV]}.
	\end{enumerate}
The generalized quasi-geostrophic   equation \eqref{gqg}
attracted a lot of attention and important progress has been made (see e.g.
\cite{[CCW],[CIW],[Kiselev],[CCCGW],[MX],[MX2],[XZ],[YZJ],[Ye]}). The goal of this paper is to examine the relationships between critical regularity for weak solutions keeping  energy (helicity) conservation and the regularity of velocity  in 2-D generalized  quasi-geostrophic equation.
A classical question involving  energy conservation in incompressible  fluid is the Onsager conjecture. In particular,
  Onsager   \cite{[Onsager]}  conjectured that the weak solutions of  incompressible  Euler equations \eqref{Euler} with
  H\"older continuity exponent $\alpha>\frac{1}{3}$ do conserve energy.
 In \cite{[CET]},
Constantin-E-Titi  successfully solved  this
   Onsager's conjecture, where it is shown that the  energy    is  conserved if a weak solution $v$ is   in the Besov space $L^{3}(0,T; B^{\alpha}_{3,\infty}(\mathbb{T}^{3}))$ with $\alpha>1/3$.  Subsequently, the result due to
 Cheskidov-Constantin-Friedlander-Shvydkoy \cite{[CCFS]}  complemented the one of
  Constantin-E-Titi  by studying in  the
   critical space $L^{3}(0,T; B^{1/3}_{3,c(\mathbb{N})})$,  where
$
B^{1/3}_{3,c(\mathbb{N})}=\{v\in B^{1/3}_{3,\infty}, \lim_{q\rightarrow\infty}2^{q}\|\Delta_{q}v\|^{3}_{L^{3}}=0\}$ and  $\Delta_{q}$ stands for a smooth restriction of $v$ into Fourier modes of order $2^q$.  The spaces $B^{1/3}_{3,c(\mathbb{N})}$ is usually called as the Onsager's    critical  spaces.
 Along this direction, there are some progresses recently, one can refer to \cite{[FW2018],[BGSTW]} for details.

 We turn our attention back to the persistence of energy in surface quasi-geostrophic equation.
A parallel of Constantin-E-Titi's result for the  2-D standard surface quasi-geostrophic equation \eqref{qg1} was obtained by Zhou  in \cite{[Zhou]}, where he showed that the $L^2$ type energy norm of weak solutions is conserved provided $\theta \in L^3(0,T;B^\alpha _{3,\infty})$ with $\alpha >\frac{1}{3}$.   Chae \cite{[Chae]} proved that the $L^{p}$ type energy norm of $\theta$ is preserved if the weak solution $(\theta,v)$ satisfy
\be\label{chae}
v \in  L^{r_1}(0,T;\dot{B}^{\alpha}_{p+1,\infty})~  \text{and}~   \theta\in  L^{r_2}(0,T;{B}^{\alpha}_{p+1,\infty}),\frac{1}{r_1}+\frac{p}{r_2}=1, \alpha>\f13.
\ee
Very recently,
Akramova-Wiedemann \cite{[AW]}   present the  following  sufficient conditions  implying    $L^{p}$ norm conservation
$$\theta\in  L^{p_1}(0,T;\dot{B}^{\alpha}_{3,\infty}), ~\alpha>\f13,~ p_1\leq\f{6}{2-3\alpha},$$
  for    a weak solution for the  2-D  standard surface
quasi-geostrophic equation \eqref{qg1}. We note that all above results are in Onsager's subcritical space other than the Onsager's critical space, which means the regularity of space is required to satisfy  $\alpha >\frac{1}{3}$ not $\alpha =\frac{1}{3}$ exactly. Hence, our first objective is to show the regularity criterion for the energy conservation of weak solutions of 2-D generalized surface quasi-geostrophic equation \eqref{gqg} in Onsager's critical space. Now, we formulate our first  result as follows.
\begin{theorem}\label{the1.1}
Let $p\in [2,\infty)$ and $\theta\in C([0,T]; L^{p}(\mathbb{R}^{2}))$ is  a  weak solution of the 2-D generalized surface quasi-geostrophic   equation   \eqref{gqg}  in the sense of Definition \ref{qgdefi},
then the $L^{p}$ type energy norm of $\theta$ is preserved, that is, for any $t\in [0,T]$,
$$\|\theta(x,t)\|_{L^{p}(\mathbb{R}^{2})}=\|\theta(x,0)\|_{L^{p}(\mathbb{R}^{2})},$$
provided one of the following conditions is satisfied
\begin{equation}\label{1.5}
	\theta\in  L^{p+1}(0,T; {{B}}^{\frac{\gamma}{3}}_{p+1,c(\mathbb{N})})~\text{with}~0<\gamma <\frac{3}{2};
\end{equation}
	\text{or}
	\begin{equation}\label{1.6}\theta\in  L^{p+1}(0,T; {{B}}^{\alpha}_{p+1,\infty})~\text{for any}~\gamma-1<\alpha<1 \text{ with} ~\frac{3}{2}\leq \gamma <2.\end{equation}

\end{theorem}
\begin{remark}
An analogue of Cheskidov-Constantin-Friedlander-Shvydkoy' s theorem \cite{[CCFS]} is established for the 2-D generalized surface quasi-geostrophic   equation \eqref{gqg}. Even though  for the standard 2-D surface quasi-geostrophic    equation \eqref{qg1},  a special case of this theorem with $p=2$ and $\gamma =1$ is novel and improves the corresponding result  in  \cite{[Zhou]}.
This theorem reveals how the regularity of the velocity  field influences the critical regularity of the weak solutions preserving the energy in generalized surface quasi-geostrophic   equation \eqref{gqg}.
\end{remark}
\begin{remark}
As pointed in \cite{[CCCGW]}, the situation in model  \eqref{gqg}    for  $1<\gamma\leq2$ is more singular than the classical quasi-geostrophic   equation \eqref{qg1}, hence we only get the subcritical criterion for energy conservation for $\frac{3}{2}\leq\gamma<2$.   It is an interesting problem to study the  persistence of energy  in \eqref{gqg}   in Onsager's critical space for the case $\frac{3}{2}\leq \gamma \leq 2$.
\end{remark}
Moreover, when $p=2$, the condition $\theta\in L^{3}(0,T;L^{3}(\mathbb{R}^{2}))$ in  Theorem \ref{the1.1} can be removed. Precisely, we have
\begin{coro}\label{coro1.1}
Let  $0<\gamma <\frac{3}{2}$.
  Assume that  $\theta\in C([0,T];L^2(\mathbb{R}^2))$ is  a  weak solution of the 2D  quasi-geostrophic   equation   \eqref{qg1}  in the sense of Definition \ref{qgdefi}    satisfying $ \theta\in  L^{3}(0,T;\dot{B}^{\frac{\gamma}{3}}_{3,c(\mathbb{N})}),$
then the $L^{2}$ type energy norm of $\theta$ is preserved, that is, for any $t\in [0,T]$,
$$\|\theta(x,t)\|_{L^{2}(\mathbb{R}^{2})}=\|\theta(x,0)\|_{L^{2}(\mathbb{R}^{2})}.$$
\end{coro}
Inspired by the  persistence of energy  criterion \eqref{chae}, we have
\begin{theorem}\label{the1.2}
Let $p\in [2,\infty)$ and $r_1\in [1,\infty], r_2\in [p,\infty]$ be given, satisfying $\frac{1}{r_1}+\frac{p}{r_2}=1$. Assume that  $\theta\in C([0,T];L^{p}(\mathbb{R}^{2}))$ is  a  weak solution of the 2-D   SQG equation \eqref{gqg}  in the sense of Definition \ref{qgdefi} with $v \in  L^{r_1}(0,T;\dot {B}^{\frac{1}{3}}_{p+1,c(\mathbb{N})})$ and $ \theta\in  L^{r_2}(0,T;{B}^{\frac{1}{3}}_{p+1,\infty})$,
then the $L^{p}$ type energy norm of $\theta$ is preserved, that is, for any $t\in [0,T]$,
$$\|\theta(t)\|_{L^{p}(\mathbb{R}^{2})}=\|\theta(0)\|_{L^{p}(\mathbb{R}^{2})}.$$
\end{theorem}
\begin{remark}
	A slightly modified the proof of this theorem means that the same result also holds if $v \in  L^{r_1}(0,T;\dot{B}^{\frac{1}{3}}_{p+1,\infty})$ and $ \theta \in L^{r_2}(0,T;{B}^{\frac{1}{3}}_{p+1,c(\mathbb{N}) })$.  This  and Theorem \ref{the1.2} refine the    criterion \eqref{chae}
\end{remark}
\begin{remark}
Owing to the boundedness of Riesz transforms in homogeneous Besov spaces,
	  Theorem \ref{the1.2} guarantees that   the $L^2$ type  energy norm   of weak solutions  satisfying
$ \theta\in  L^{3}(0,T;\dot{B}^{\frac{1}{3}}_{3,c(\mathbb{N})})$
are preserved in the standard   quasi-geostrophic   equation \eqref{qg1}.
\end{remark}
We would like to mention that  Dai \cite{[Dai]}  showed that the energy of any viscosity solution of the     quasi-geostrophic equation \eqref{qg1} with supercritical   dissipation $\Lambda^{\alpha}\theta$
satisfying  $\theta \in L^{2}(0,T;B_{2,c(\mathbb{N})}^{\f{1}{2}})$  is   invariant.

Beside the  persistence of energy of weak solutions in   the quasi-geostrophic   equation \eqref{qg1}, the general helicity defined as
$$
\int  \theta \partial_{i}\theta   dx, i=1,2,
$$
is  also conserved,  which is observed by Zhou  in \cite{[Zhou]}. It is shown that  the general helicity of weak solutions for the 2-D standard surface quasi-geostrophic equation \eqref{qg1} is conserved  if
$\nabla\theta \in C([0,T];L^{\f43}(\mathbb{R}^{2}))\cap L^{3}(0,T;B^{\alpha}_{\f32,\infty})$ with $\alpha>1/3$ in \cite{[Zhou]}. Recently, the authors \cite{[WWY]} improves this  to $\nabla\theta \in C([0,T];L^{\f43}(\mathbb{R}^{2}))\cap L^{3}(0,T;\dot{B}^{\frac{1}{3}}_{\f32,c(\mathbb{N})})$.
The helicity  of flow was    originated from  Moffatt's work     \cite{[Moffatt]}. The helicity is important at a fundamental level in relation to flow kinematics
because it admits topological interpretation in relation to the linkage or
linkages of vortex lines of the flow (see \cite{[Moffatt],[MT]}). See \cite{[Chae],[Chae1],[De Rosa],[WWY]} for the study of the
helicity conservation of weak solutions of the Euler equations.
Now  we state our rest results involving helicity conservation of weak solutions for the generalized    quasi-geostrophic equation \eqref{gqg} as follows:

\begin{theorem}\label{the1.3} Let  $\theta$ be a  weak solution of the 2-D generalized quasi-geostrophic equation  \eqref{gqg} in the sense of Definition \ref{qgdefi}, then the helicity  conservation
\be\label{ghqg}
\int_{\mathbb{R}^{2}}   \theta(x,t) \partial_{i}\theta(x,t)   dx =\int_{\mathbb{R}^{2}}  \theta_0(x) \partial_{i}\theta_0 (x) dx, i=1,2, \ee
 is valid provided one of the following conditions is
satisfied
 \begin{enumerate}[(1)] \item
  $ \nabla\theta \in L^{3} (0,T;\dot{B}^{\f\gamma3}_{\f32,c(\mathbb{N})} (\mathbb{R}^{2}))\cap C([0,T];L^{\frac{4}{3}}(\mathbb{R}^{2})) $ ~with~$0<\gamma <\frac{3}{2};
$	 \item
$\nabla\theta\in  L^{3}(0,T; \dot{B}^{\alpha}_{\f32,\infty}(\mathbb{R}^{2}))\cap C([0,T];L^{\frac{4}{3}}(\mathbb{R}^{2}))$~ for any ~$\gamma-1<\alpha<1$  with  $\frac{3}{2}\leq \gamma <2.$
 \end{enumerate}
\end{theorem}
\begin{remark}
	By the Bernstein inequality,    $ \nabla\theta \in L^{3} (0,T;\dot{B}^{\f\gamma3}_{\f32,c(\mathbb{N})}(\mathbb{R}^{2}) ) $ may be replaced by $  \theta \in L^{3} (0,T;\dot{B}^{\f{3+\gamma}{3}}_{\f32,c(\mathbb{N})} (\mathbb{R}^{2})) $ in this theorem.
\end{remark}
\begin{remark}
	The required regularity $\nabla \theta \in C([0,T];L^{\frac{4}{3}}(\mathbb{R}^{2}))$ is used to ensure the helicity conservation \eqref{ghqg} make sense.
\end{remark}
\begin{remark}
Compare the results in Theorem \ref{the1.3} and Corollary \ref{coro1.1}, we see that there may exist   a weak solution of the 2-D generalized surface quasi-geostrophic equation \eqref{gqg} that keep  the  $L^2$ type energy   rather than the helicity. This was also previously
	pointed out by Chae in  \cite{[Chae]} for 2-D standard surface quasi-geostrophic equation \eqref{qg1}.
\end{remark}
\begin{remark}
	Our result in Theorem \ref{the1.3} covers and generalizes the recent result obtained for 2-D standard surface quasi-geostrophic equation \eqref{qg1} in \cite{[WWY]}.
\end{remark}
 We will provide two approaches to show Theorem \ref{the1.1}-\ref{the1.3}. One is an appliction of the Littlewood-Paley theory developed  by  Cheskidov-Constantin-Friedlander-Shvydkoy in \cite{[CCFS]}. The second one relies on  the
       Constantin-E-Titi type commutator estimates  in  physical Onsager type spaces (see Lemma \ref{lem2.3}). It seems that the arguments in this paper can be applicable to other fluid models such as the  surface growth model
without dissipation
\begin{equation}\label{sgm}
h_{t} +\partial_{xx}(h_{x})^{2}= 0,
\end{equation}
where $h$ stands for the height of a crystalline layer.
The background of the  surface growth model \eqref{sgm} can be found in \cite{[WYM],[O],[OR],[BR2009]}. The  energy conservation in the Besov space $L^{3}(0,T; B^{\alpha}_{3,\infty}(\mathbb{T}^{3}))$ with $\alpha>1/3$
was considered in \cite{[WYM]}. One can establish the persistence of energy  criterion
 in the Onsager's    critical  spaces for the inviscid surface growth model \eqref{sgm}.

The rest of the paper is organized as follows.
 In Section 2, we present some notations and  auxiliary lemmas which will be  frequently used throughout this paper.
 The  energy conservation of weak solutions of the  surface quasi-geostrophic equation  is  considered  in  Section 3.
 Section 4 is devoted to   the helicity conservation of weak solutions of the   generalized surface quasi-geostrophic equation. Concluding remarks are given in Section 5.

\section{Notations and some auxiliary lemmas} \label{section2}

{\bf Sobolev spaces:} First, we introduce some notations used in this paper.
 For $p\in [1,\,\infty]$, the notation $L^{p}(0,\,T;X)$ stands for the set of measurable functions on the interval $(0,\,T)$ with values in $X$ and $\|f(t,\cdot)\|_{X}$ belonging to $L^{p}(0,\,T)$. The classical Sobolev space $W^{k,p}(\mathbb{R}^2)$ is equipped with the norm $\|f\|_{W^{k,p}(\mathbb{R}^2)}=\sum\limits_{|\alpha| =0}^{k}\|D^{\alpha}f\|_{L^{p}(\mathbb{R}^2)}$. \\
  {\bf Besov spaces:} $\mathcal{S}$ denotes the Schwartz class of rapidly decreasing functions, $\mathcal{S}'$ the
space of tempered distributions, $\mathcal{S}'/\mathcal{P}$ the quotient space of tempered distributions which modulo polynomials.
  We use $\mathcal{F}f$ or $\widehat{f}$ to denote the Fourier transform of a tempered distribution $f$.
To define Besov  spaces, we need the following dyadic unity partition
(see e.g. \cite{[BCD]}). Choose two nonnegative radial
functions $\varrho$, $\varphi\in C^{\infty}(\mathbb{R}^{d})$
supported respectively in the ball $\mathcal{B}=\{\xi\in
\mathbb{R}^{d}:|\xi|\leq \frac{3}{4} \}$ and the shell $\mathcal{C}\{\xi\in
\mathbb{R}^{d}: \frac{3}{4}\leq |\xi|\leq
  \frac{8}{3} \}$ such that
\begin{equation*}
 \varrho(\xi)+\sum_{j\geq 0}\varphi(2^{-j}\xi)=1, \quad
 \forall\xi\in\mathbb{R}^{d}; \qquad
 \sum_{j\in \mathbb{Z}}\varphi(2^{-j}\xi)=1, \quad \forall\xi\neq 0.
\end{equation*}
Then for every $\xi\in\mathbb{R}^{d},$ $\varphi(\xi)=\varrho(\xi/2)-\varrho(\xi)$. Write $h=\mathcal{F}^{-1} \varphi $ and $\tilde{h}=\mathcal{F}^{-1}\varrho$, then nonhomogeneous dyadic blocks  $\Delta_{j}$ are defined by
$$
\Delta_{j} u:=0 ~~ \text{if} ~~ j \leq-2, ~~ \Delta_{-1} u:=\varrho(D) u =\int_{\mathbb{R}^d}\tilde{h}(y)u(x-y)dy,$$
$$\text{and}~~\Delta_{j} u:=\varphi\left(2^{-j} D\right) u=2^{jd}\int_{\mathbb{R}^d}h(2^{j}y)u(x-y)dy  ~~\text{if}~~ j \geq 0.
$$
The nonhomogeneous low-frequency cut-off operator $S_j$ is defined by
$$
S_{j}u:= \sum_{k\leq j-1}\Delta_{k}u=\varrho(2^{-j}D)u=2^{jd}\int_{\mathbb{R}^d}\tilde{h}(2^{j}y)u(x-y)dy, ~j\in \mathbb{N}\cup {0}.$$
The homogeneous dyadic blocks $\dot{\Delta}_{j}$ and homogeneous low-frequency cut-off operators $\dot{S}_j$ are  defined  for $ \forall j\in\mathbb{Z}$ by
\begin{equation*}
  \dot{\Delta}_{j}u:= \varphi(2^{-j}D)u=2^{jd}\int_{\mathbb{R}^d}h(2^{j}y)u(x-y)dy,~j\in \mathbb{Z}
\end{equation*}
$$ \text { and }~~ \dot{S}_{j}u:=\varrho(2^{-j}D)u=2^{jd}\int_{\mathbb{R}^d}\tilde{h}(2^{j}y)u(x-y)dy,~j\in \mathbb{Z}$$
Now we introduce the definition of Besov spaces. Let $(p, r) \in[1, \infty]^{2}, s \in \mathbb{R}$, the nonhomogeneous Besov space
$$
B_{p, r}^{s}:=\left\{f \in \mathcal{S}^{\prime}\left(\mathbb{R}^{d}\right) ;\|f\|_{B_{p, r}^{s}}:=\left\|2^{j s}\right\| \Delta_{j} f\left\|_{L^{p}}\right\|_{\ell^{r}(\mathbb{Z})}<\infty\right\}
$$
and the homogeneous space
$$
\dot{B}_{p, r}^{s}:=\left\{f \in \mathcal{S}^{\prime}\left(\mathbb{R}^{d}\right) / \mathcal{P}\left(\mathbb{R}^{d}\right) ;\|f\|_{\dot{B}_{p, r}^{s}}:=\left\|2^{j s}\right\| \dot{\Delta}_{j} f\left\|_{L^{p}}\right\|_{\ell^{r}(\mathbb{Z})}<\infty\right\} .
$$
Moreover, for $s>0$ and $1\leq p,q\leq \infty$, we may write the
equivalent norm below in
  the nonhomogeneous Besov norm $\norm{f}_{B^s_{p,q}}$ of $f\in \mathcal{S}^{'}$ as
$$\norm{f}_{B^s_{p,q}}=\norm{f}_{{L^p}}+\norm{f}_{\dot{B}^s_{p,q}}.$$
Motivated by \cite{[CCFS]}, we define $\dot{B}^\alpha _{p,c(\mathbb{N})}$ to be the class of all tempered distributions $f$ for which
\begin{equation}\label{2.1}
\norm{f}_{\dot{B}^\alpha _{p,\infty}}<\infty~ \text{and}~ 	\lim_{j\rightarrow \infty} 2^{j\alpha}\norm{\dot{\Delta}_j f}_{L^p}=0,~~\text{for any}~1\leq p\leq \infty.
\end{equation}
It is clear that the Besov
spaces $\dot{B}^\alpha_{p,q}$ are included in $\dot{B}^\alpha_{p,c(\mathbb{N})}$ for any $1\leq q< \infty$. Likewise, one can define the Besov
spaces $ {B}^\alpha_{p,c(\mathbb{N})}$ similarly.\\
{\bf Mollifier kernel:} Let $\eta_{\varepsilon}:\mathbb{R}^{d}\rightarrow \mathbb{R}$ be a standard mollifier.i.e. $\eta(x)=C_0e^{-\frac{1}{1-|x|^2}}$ for $|x|<1$ and $\eta(x)=0$ for $|x|\geq 1$, where $C_0$ is a constant such that $\int_{\mathbb{R}^d}\eta (x) dx=1$. For $\varepsilon>0$, we define the rescaled mollifier $\eta_{\varepsilon}(x)=\frac{1}{\varepsilon^d}\eta(\frac{x}{\varepsilon})$ and for  any function $f\in L^1_{loc}(\mathbb{R}^d)$, its mollified version is defined as
$$f^\varepsilon(x)=(f*\eta_{\varepsilon})(x)=\int_{\mathbb{R}^d}f(x-y)\eta_{\varepsilon}(y)dy,\ \ x\in \mathbb{R}^d.$$
Next, we collect some Lemmas which will be used in the present paper.
\begin{lemma}(Bernstein inequality \cite{[BCD]})\label{berinequ}  Let $\mathcal{B}$ be a ball of $\mathbb{R}^{d}$, and $\mathcal{C}$ be a ring of $\mathbb{R}^{d}$. There exists a positive constant $C$ such that for all integer $k \geq 0$, all $1 \leq a \leq b \leq \infty$ and $u \in L^{a}\left(\mathbb{R}^{d}\right)$, the following estimates are satisfied:
$$
\begin{gathered}
\sup _{|\alpha|=k}\left\|\partial^{\alpha} u\right\|_{L^{b}\left(\mathbb{R}^{d}\right)} \leq C^{k+1} \lambda^{k+d\left(\frac{1}{a}-\frac{1}{b}\right)}\|u\|_{L^{a}\left(\mathbb{R}^{d}\right)}, \quad \operatorname{supp} \hat{u} \subset \lambda \mathcal{B}, \\
C^{-(k+1)} \lambda^{k}\|u\|_{L^{a}\left(\mathbb{R}^{d}\right)} \leq \sup _{|\alpha|=k}\left\|\partial^{\alpha} u\right\|_{L^{a}\left(\mathbb{R}^{d}\right)} \leq C^{k+1} \lambda^{k}\|u\|_{L^{a}\left(\mathbb{R}^{d}\right)}, \quad \operatorname{supp} \hat{u} \subset \lambda \mathcal{C}.
\end{gathered}
$$
\end{lemma}

\begin{lemma}(\cite{[WWY]})\label{lem2.2}
Let $\Omega$ denote the whole space $\mathbb{R}^{d}$ or the periodic domain $\mathbb{T}^{d}$.
Suppose that $\alpha, \beta\in (0,1)$,  $ p,q\in [1,\infty]$,  and $k\in \mathbb{N}^+$. Assume that  $f\in L^p(0,T;\dot{B}^\alpha_{q,\infty})$, $g\in L^p(0,T;\dot{B}^\beta_{q,c(\mathbb{N})})$,  then there holds that
 \begin{enumerate}[(1)]
 \item $ \|f^{\varepsilon} -f \|_{L^{p}(0,T;L^{q}(\Omega))}\leq C\text{O}(\varepsilon^{\alpha})\|f\|_{L^p(0,T;\dot{B}^\alpha_{q,\infty})}$;
   \item   $ \|\nabla^{k}f^{\varepsilon}  \|_{L^{p}(0,T;L^{q}(\Omega))}\leq C\text{O}(\varepsilon^{\alpha-k})\|f\|_{L^p(0,T;\dot{B}^\alpha_{q,\infty})}$;
       \item $ \|g^{\varepsilon} -g \|_{L^{p}(0,T;L^{q}(\Omega))}\leq C\text{o}(\varepsilon^{\beta})\|g\|_{L^p(0,T;\dot{B}^\beta_{q,c(\mathbb{N})})}$;
   \item   $ \|\nabla^{k}g^{\varepsilon}  \|_{L^{p}(0,T;L^{q}(\Omega))}\leq C\text{o}(\varepsilon^{\beta-k})\|g\|_{L^p(0,T;\dot{B}^\beta_{q,c(\mathbb{N})})}$;
 \end{enumerate}
\end{lemma}

Next, we will state the Constantin-E-Titi type commutator estimates in  physical Onsager type spaces (see also \cite{[Yu]}).
\begin{lemma}(\cite{[WWY]})	\label{lem2.3}
Let $\Omega$ denote the whole space $\mathbb{R}^{d}$ or the periodic domain $\mathbb{T}^{d}$.
	Assume that $0<\alpha,\beta<1$, $1\leq p,q,p_{1},p_{2}\leq\infty$ and $\frac{1}{p}=\frac{1}{p_1}+\frac{1}{p_2}$.
Then, there holds
	\begin{align} \label{cet}
		\|(fg)^{\varepsilon}- f^{\varepsilon}g^{\varepsilon}\|_{L^p(0,T;L^q(\Omega))} \leq C\text{o}(\varepsilon^{\alpha+\beta}),	
	\end{align}
provided one of the following three conditions holds
\begin{enumerate}[(1)]
 \item  $f\in L^{p_1}(0,T;\dot{B}^{\alpha}_{q_{1},c(\mathbb{N})} )$, $g\in L^{p_2}(0,T;\dot{B}^{\beta}_{q_{2},\infty} )$,$1\leq q_{1},q_{2}\leq\infty$ $\frac{1}{q}=\frac{1}{q_1}+\frac{1}{q_2}$;
  \item  $\nabla f\in   L^{p_1}(0,T;\dot{B}^{\alpha}_{q_{1},c(\mathbb{N})} )$, $\nabla g\in L^{p_2}(0,T;\dot{B}^{\beta}_{q_{2},\infty} )$,  $\f{2}{d}+\f1q=\frac{1}{q_{1}}+\frac{1}{q_{2}}$,$1\leq q_{1},q_{2}<d$;
  \item  $  f\in   L^{p_1}(0,T;\dot{B}^{\alpha}_{q_{1},c(\mathbb{N})} )$, $\nabla g\in L^{p_2}(0,T;\dot{B}^{\beta}_{q_{2},\infty} )$,  $\f{1}{d}+\f1q=\frac{1}{q_{1}}+\frac{1}{q_{2}}$,$1\leq q_{2}<d$,  $1\leq q_{1}\leq\infty$.
 \end{enumerate}\end{lemma}

For the convenience of readers, we present the definition of the weak solutions of   the surface quasi-geostrophic   equation
 \eqref{gqg}.

\begin{definition}\label{qgdefi}
	A vector field $\theta\in C_{\text{weak}}([0,T];L^{p}(\mathbb{R}^{2}))$ is called  a weak solution of the 2-D  quasi-geostrophic equation with initial data $\theta_{0}\in L^{p}(\mathbb{R}^{2})$ with $p\in [2,\infty)$ if there holds	
\begin{equation}
	\int_{\mathbb{R}^{2}}[\theta(x,t) \varphi                                     (x,t)- \theta(x,0) \varphi(x,0)]dx =\int_{0}^{t}\int_{\mathbb{R}^{2}}\theta(x,s)\big(\partial_{t}\varphi(x,s)+v(x,s)\cdot\nabla\varphi(x,s)\big)dxds
\end{equation}
	and
	\begin{equation}\label{1.9}
		v(x,t)=\mathcal{R}^{\perp}\Lambda^{\gamma-1}\theta,
	\end{equation}
	for any   test function $\varphi\in C_{0}^{\infty}([0,T];C^{\infty}(\mathbb{R}^{2}))$.
\end{definition}

\section{Energy conservation of weak solutions for 2-D surface quasi-geostrophic equation}
In this section, we are concerned with the energy conservation
for 2-D generalized surface quasi-geostrophic equation \eqref{gqg} and 2-D standard surface quasi-geostrophic equation \eqref{qg1}.  To prove theorem \ref{the1.1}, we will give two different approaches due to Littlewood-Paley theory developed  by  Cheskidov-Constantin-Friedlander-Shvydkoy in \cite{[CCFS]}and the
Constantin-E-Titi type commutator estimates  in  physical Onsager type spaces (see Lemma \ref{lem2.3}), respectively.
\begin{proof}[Proof of Theorem \ref{the1.1}] \  \\
{\bf Approach 1: Littlewood-Paley theory:} Multiplying  the surface quasi-geostrophic equation \eqref{gqg}
by $S_{N}(S_{N}\theta |S_{N}\theta |^{p-2}) $ with $p\geq 2$ (see the notations in Section 2), together with the incompressible condition and using integration by parts, we see that
$$
\f{1}{p}\f{d}{dt}\int_{\mathbb{R}^{2}} |S_{N}\theta |^{p}dx=(p-1)\int_{\mathbb{R}^{2}} S_{N}(v_{j}\theta)  \partial_{j} S_{N}\theta |S_{N}\theta|^{p-2}dx.
$$
Since the divergence-free condition of the velocity field $v(x,t)$ helps us to derive that
$$
\int_{\mathbb{R}^{2}} S_{N}v_{j} \partial_{j}S_{N}\theta
S_{N}\theta |S_{N}\theta|^{p-2}dx=0,
$$
thus we conclude that
$$
\f{1}{p}\f{d}{dt}\int_{\mathbb{R}^{2}} |S_{N}\theta|^{p}dx=(p-1)\int_{\mathbb{R}^{2}}\big[S_{N}(v_{j}\theta)  -S_{N}v_{j} S_{N}\theta \big]\partial_{j} S_{N}\theta| S_{N}\theta |^{p-2}dx.
$$
Recall the Constantin-E-Titi identity
\be\ba\label{ceti}
&S_{N}(fg)  -S_{N}f S_{N}g\\
=&  2^{ 2N}\int_{\mathbb{R}^2}\tilde{h}(2^{N}y)[f(x-y)-f(x)][g(x-y)-g(x)]dy-(f-S_{N}f)(g-S_{N}g),
\ea\ee
where we used  $2^{ 2N}\int_{\mathbb{R}^2}\tilde{h}(2^{N}y)dy=\mathcal{F}(\tilde{h}(\cdot))|_{\xi=0}=1$.

Taking advantage of the H\"older inequality, we discover that
\begin{equation}\label{key03}
	\begin{aligned}
	&\B| \int_{\mathbb{R}^{2}}\big[S_{N}(v_{j}\theta) -S_{N}v_{j} S_{N}\theta \big]\partial_{j} S_{N}\theta| S_{N}\theta|^{p-2}dx \B|\\
\leq	&C  \|S_{N}(v_{j}\theta)  -S_{N}v_{j}S_{N}\theta\|_{ L^{\f{p+1}{2}} (\mathbb{R}^{2})}  \|\partial_{j}S_{N}\theta\|_{ L^{p+1}  (\mathbb{R}^{2})} \||S_{N}\theta|^{p-2}\|_{ L^{\f{p+1}{p-2} }  (\mathbb{R}^{2})}\\
\leq	&C  \|S_{N}(v_{j}\theta)  -S_{N}v_{j}S_{N}\theta\|_{ L^{\f{p+1}{2}}  (\mathbb{R}^{2})}  \|\partial_{j}S_{N}\theta\|_{ L^{p+1}  (\mathbb{R}^{2})} \| S_{N}\theta \|^{p-2}_{L^{p+1} (\mathbb{R}^{2})}.
\end{aligned}\end{equation}
With the help of
\eqref{ceti} and the Minkowski inequality, we write
$$\ba\label{key}
 &\|S_{N}(v_{j}\theta)  -S_{N}v_{j}S_{N}\theta\|_{  L^{\f{p+1}{2}}  (\mathbb{R}^{2})}
 \\\leq& 2^{ 2N}\int_{\mathbb{R}^2}|\tilde{h}(2^{N}y)|\| v_{j}(x-y)-v_{j}(x) \|_{ L^{p+1}  (\mathbb{R}^{2})}
 \| \theta(x-y)-\theta(x) \|_{ L^{p+1} (\mathbb{R}^{2}) }dy
\\&+\| v_{j}-S_{N}v_{j} \|_{ L^{p+1}  (\mathbb{R}^{2})} \| \theta-S_{N}\theta \|_{L^{p+1} (\mathbb{R}^{2})}\\
=&I+II.
\ea$$
In view of  the mean value theorem and the Bernstein inequality, we know that
\be\ba\label{e3.3}
&\| v_{j}(x-y)-v_{j}(x) \|_{ L^{p+1} (\mathbb{R}^{2})}
\leq&C\B(\sum_{j\leq N}2^{j}|y|\|\dot{\Delta}_{j}v\|_{L^{p+1} (\mathbb{R}^{2})}+\sum_{j>N}\|\dot{\Delta}_{j}v\|_{L^{p+1} (\mathbb{R}^{2})}\B).
\ea\ee
Using the Bernstein inequality again and the boundedness of Riesz transforms on Lebesgue spaces, we see that
$$
\|\dot{\Delta}_jv\|_{L^{p+1} (\mathbb{R}^{2})}=\|\mathcal{R}^{\perp}\Lambda^{\gamma-1}\dot{\Delta}_j\theta\|_{L^{p+1} (\mathbb{R}^{2})}\leq C2^{j(\gamma-1)}\|\dot{\Delta}_j\theta\|_{L^{p+1} (\mathbb{R}^{2})},~\text{for} ~0<p<\infty. $$
This together with \eqref{e3.3} means that
\be\ba\label{3.4}
&\| v_{j}(x-y)-v_{j}(x) \|_{ L^{p+1}  (\mathbb{R}^{2})} \\
\leq&C\B( 2^{N(\gamma-\alpha)}|y|\sum_{j\leq N}2^{-(N-j)(\gamma-\alpha)}2^{j\alpha}\|\dot{\Delta}_{j}\theta\|_{L^{p+1} (\mathbb{R}^{2})}\\&+2^{(\gamma-1-\alpha )N}\sum_{j>N} 2^{(N-j)(\alpha+1-\gamma)} 2^{j\alpha}\|\dot{\Delta}_{j}\theta\|_{L^{p+1} (\mathbb{R}^{2})}\B).
\ea\ee
Before going further, in the spirit of  \cite{[CCFS]}, we set the following localized kernel
\be\label{K1}K_{1}(j)=\left\{\begin{aligned}
	&2^{j(\alpha+1-\gamma)},~~~~~~~~\text{if}~j\leq0,\\
	& 2^{-(\gamma-\alpha)j},~~~~~~~~~\text{if}~j>0,
\end{aligned}\right.
\ee
and  we denote $\dot{d}_j=2^{j\alpha}\|\dot{\Delta}_{j}\theta\|_{L^{p+1} (\mathbb{R}^{2})}$.\\
As a consequence, we get
$$\ba
\| v_{j}(x-y)-v_{j}(x) \|_{ L^{p+1} (\mathbb{R}^{2}) }\leq& C\left[ 2^{N(\gamma-\alpha)}|y| +2^{(\gamma-1-\alpha )N}\right]\left(K_{1}\ast \dot{d}_j\right)(N)\\
\leq & C( 2^{N}|y| +1)2^{(\gamma-1-\alpha )N}\left(K_{1}\ast \dot{d}_j\right)(N).
 \ea$$
To bound $\| \theta(x-y)-\theta(x) \|_{ L^{p+1}  (\mathbb{R}^{2})}$, just as \cite{[CCFS]}, we denote
\be\label{K2}
K_{2}(j)=\left\{\begin{aligned}
	&2^{j\alpha},~~~~~~~~\text{if}~j\leq0,\\
	& 2^{-(1-\alpha)j},~~\text{if}~j>0.
\end{aligned}\right.
\ee
A slightly modified proof of \eqref{e3.3} and \eqref{3.4} gives
\be\ba\label{3.5}
&\| \theta(x-y)-\theta(x) \|_{ L^{p+1}  (\mathbb{R}^{2})}\\
\leq &C\B(\sum_{j\leq N}2^{j}|y|\|\dot{\Delta}_{j}\theta\|_{L^{p+1} (\mathbb{R}^{2})}+\sum_{j>N}\|\dot{\Delta}_{j}\theta\|_{L^{p+1} (\mathbb{R}^{2})}\B)\\
\leq&C\B( 2^{N(1-\alpha)}|y|\sum_{j\leq N}2^{-(N-j)(1-\alpha)}2^{j\alpha}\|\dot{\Delta}_{j}\theta\|_{L^{p+1} (\mathbb{R}^{2})}+2^{-\alpha N}\sum_{j>N} 2^{(N-j)\alpha} 2^{j\alpha}\|\dot{\Delta}_{j}\theta\|_{L^{p+1} (\mathbb{R}^{2})}\B)
\\
\leq& C\left[ 2^{N(1-\alpha)}|y| +2^{ -\alpha N}\right]\left(K_{2}\ast \dot{d}_j\right)(N)\\
\leq& C(2^{N}|y| +1)2^{ -\alpha N}\left(K_{2}\ast \dot{d}_j\right)(N).
\ea\ee
Notice that
$$ \sup_{N} 2^{ 2N}\int_{\mathbb{R}^2}  |\tilde{h}(2^{N}y)|(2^{N}|y| +1)^{2}dy<\infty.$$
Hence, we deduce from  \eqref{3.4} and  \eqref{3.5} that
$$I\leq C2^{(\gamma-1-\alpha )N}\left(K_{1}\ast \dot{d}_j\right)(N)2^{ -\alpha N}\left(K_{2}\ast \dot{d}_j\right)(N).$$
In light of
 the Bernstein inequality, we infer that
$$\| v_{j}-{S}_{N}v_{j} \|_{ L^{p+1}  (\mathbb{R}^{2})} \leq
\sum_{j\geq N}\|\dot{\Delta}_{j}v\|_{L^{p+1}}\leq C2^{(\gamma-1-\alpha )N}\left(K_{1}\ast  \dot{d}_j \right)(N),
$$
where  we used $N>0.$

Likewise,
$$\| \theta-{S}_{N}\theta \|_{ L^{p+1}  (\mathbb{R}^{2})}\leq C 2^{ -\alpha N}\left(K_{2}\ast \dot{d}_j\right)(N),
$$
 from which it follows that
$$
II\leq C2^{(\gamma-1-\alpha )N}\left(K_{1}\ast \dot{d}_j\right)(N)2^{ -\alpha N}\left(K_{2}\ast \dot{d}_j\right)(N).
$$
Consequently, we know that
\be\label{3.7}
\|S_{N}(v_{j}\theta)  -S_{N}v_{j}S_{N}\theta\|_{ L^{\f{p+1}{2}} (\mathbb{R}^{2}) }\leq C2^{(\gamma-1-\alpha )N}\left(K_{1}\ast \dot{d}_j\right)(N)2^{ -\alpha N}\left(K_{2}\ast \dot{d}_j\right)(N).
\ee
We conclude
by some straightforward calculations that
\be\label{3.8}
\|\partial_{j}S_{N}\theta\|_{ L^{p+1} (\mathbb{R}^{2}) }\leq \sum_{j\leq N}2^{j} \| {\Delta}_{j}\theta\|_{L^{p+1} (\mathbb{R}^{2})}\leq  2^{N(1-\alpha)}\left(K_{2}\ast {d}_j\right)(N),
\ee
where ${d}_{j}=2^{j\alpha}\|{\Delta}_{j}\theta\|_{L^{p+1} (\mathbb{R}^{2})}$.\\
Inserting \eqref{3.7} and \eqref{3.8} into  \eqref{key03} gives
\be\ba\label{e3.10}
&\|S_{N}(v_{j}\theta)  -S_{N}v_{j}S_{N}\theta\|_{ L^{\f{p+1}{2}}  (\mathbb{R}^{2})}\|\partial_{j}S_{N}\theta\|_{ L^{p+1}  (\mathbb{R}^{2})} \\
\leq& C2^{(\gamma-1-\alpha )N}\left(K_{1}\ast \dot{d}_j\right)(N)2^{ -\alpha N}\left(K_{2}\ast \dot{d}_j\right)(N) 2^{N(1-\alpha)}\left(K_{2}\ast {d}_j\right)(N)\\
\leq& C2^{(\gamma-3\alpha )N}\left(K_{1}\ast \dot{d}_j\right)(N) \left(K_{2}\ast \dot{d}_j\right)(N)\left( K_{2}\ast {d}_j\right)(N).
\ea\ee
 To ensure that $K_{1}, K_{2}\in  l^{1}(\mathbb{Z})$, we need
\be\left\{\ba\label{}
&\alpha+1-\gamma>0,\\
&\gamma-\alpha>0,\\
&0<\alpha<1,\\
&\gamma-3\alpha\leq 0,
\ea\right.\ee
which lead to $\alpha\geq\f{\gamma}{3}$ and $\alpha<\gamma<\alpha+1$.

Then substituting \eqref{e3.10} into \eqref{key03} and using the Young inequality, we arrive at
\be\ba\label{3.111}
&\B| \int\big[S_{N}(v_{j}\theta) -S_{N}v_{j} S_{N}\theta \big]\partial_{i} S_{N}\theta| S_{N}\theta|^{p-2}dx \B|\\
\leq&C2^{(\gamma-3\alpha )N}\left(K_{1}\ast \dot{d}_j\right)(N)\left( K_{2}\ast \dot{d}_j\right)(N)\left( K_{2}\ast {d}_j\right)(N)\| S_{N}\theta \|^{p-2}_{L^{p+1} (\mathbb{R}^{2})}\\
\leq&C2^{(\gamma-3\alpha )N}\left(K_{1}\ast \dot{d}_j\right)(N)\left( K_{2}\ast \dot{d}_j\right)(N) \left(K_{2}\ast {d}_j\right)(N)\|  \theta \|^{p-2}_{L^{p+1} (\mathbb{R}^{2})}\\
\leq &C2^{(\gamma-3\alpha )N}\sup_{N}(\dot{d}_N)^2\sup_{N}({d}_N)\|  \theta \|^{p-2}_{L^{p+1} (\mathbb{R}^{2})},
\ea\ee
where $\dot{d}_{N}=2^{N\alpha}\|\dot{\Delta}_{N}\theta\|_{L^{p+1} (\mathbb{R}^{2})}$ and ${d}_{N}=2^{N\alpha}\|{\Delta}_{N}\theta\|_{L^{p+1} (\mathbb{R}^{2})}$.

Case 1: if $\alpha=\frac{\gamma}{3}$  with $0<\gamma <\frac{3}{2}$,  it follows from \eqref{3.111} and the dominated convergence theorem that
 \be\ba
 &\B| \int_{\mathbb{R}^{2}}\big[S_{N}(v_{j}\theta) -S_{N}v_{j} S_{N}\theta \big]\partial_{i} S_{N}\theta| S_{N}\theta|^{p-2}dx \B|\\
  \leq &C\left(K_{1}\ast \dot{d}_j\right)(N)\left( K_{2}\ast \dot{d}_j\right)(N)\sup_{N}({d}_N)\|  \theta \|^{p-2}_{L^{p+1} (\mathbb{R}^{2})}\\
  \leq &C\left(K_{1}\ast \dot{d}_j\right)(N)\left( K_{2}\ast \dot{d}_j\right)(N)\|  \theta \|^{p-1}_{{{B} }^{\alpha}_{p+1,\infty }} \rightarrow 0,~\text{as} ~N\rightarrow +\infty.
 \ea\ee
This together with  $\theta\in  L^{p+1}(0,T; B^{\frac{\gamma}{3}}_{p+1,c(\mathbb{N})})$, we conclude by the
the dominated convergence theorem once again that
$$\ba
&\limsup_{N\rightarrow +\infty}\int_{0}^{T}\B| \int_{\mathbb{R}^{2}}\big[S_{N}(v_{j}\theta) -S_{N}v_{j} S_{N}\theta \big]\partial_{i} S_{N}\theta| S_{N}\theta|^{p-2}dx \B|dt
\\  \leq  &C\int_{0}^{T}\left(K_{1}\ast \dot{d}_j\right)(N)\left( K_{2}\ast \dot{d}_j\right)(N)\|  \theta \|^{p-1}_{{{B} }^{\alpha}_{p+1,\infty }}dt \rightarrow 0,
\ea$$

Case 2: if $\frac{\gamma}{3}\leq \gamma-1<\alpha<1$ with $\frac{3}{2}\leq \gamma <2$, then from \eqref{3.111} and taking  $N\rightarrow +\infty$, we have
 $$\ba
 &\B|\int_{\mathbb{R}^{2}}\big[S_{N}(v_{j}\theta)  -S_{N}v_{j} S_{N}\theta \big]\partial_{j} S_{N}\theta| S_{N}\theta |^{p-2}dx\B|\\
 \leq &C2^{(\gamma-3\alpha )N}\|  \theta \|^{3}_{{{B} }^{\alpha}_{p+1,\infty }(\mathbb{R}^{2})}\|  \theta \|^{p-2}_{L^{p+1} (\mathbb{R}^{2})}\\
 \leq &C2^{(\gamma-3\alpha )N}\|  \theta \|^{p+1}_{{{B} }^{\alpha}_{p+1,\infty }(\mathbb{R}^{2})}\rightarrow 0.
 \ea$$
 which in turn gives
 $$\ba
&\limsup_{N\rightarrow +\infty}\int_{0}^{t}\B| \int_{\mathbb{R}^{2}}\big[S_{N}(v_{j}\theta) -S_{N}v_{j} S_{N}\theta \big]\partial_{i} S_{N}\theta| S_{N}\theta|^{p-2}dx \B|ds
\\   \leq  &\limsup_{N\rightarrow +\infty} 2^{(\gamma-3\alpha )N}C\int_{0}^{t}\|  \theta \|^{p+1}_{{{B} }^{\alpha}_{p+1,\infty }(\mathbb{R}^{2})}ds\rightarrow 0.\ea$$
 Hence,  no matter in which case,  we have
 $$\ba
 &\B|\int_{0}^{t}\int_{\mathbb{R}^{2}}\big[S_{N}(v_{j}\theta)  -S_{N}v_{j} S_{N}\theta \big]\partial_{j} S_{N}\theta| S_{N}\theta |^{p-2}dxds\B|\rightarrow0,~\text{as} ~N\rightarrow +\infty.\ea$$
Then we have completed the proof of Theorem \ref{the1.1}.

\ \\
{\bf Approach 2: Constantin-E-Titi type commutator estimates  in  physical Onsager type spaces:}
  Mollifying the surface quasi-geostrophic equation \eqref{gqg} in space (see the notations in Section 2) and using the divergence-free condition, we know that
$$\theta^{\varepsilon}_{t}+\text{div} (v \theta)^{\varepsilon}=0,$$
 which yields that
 $$
\f{1}{p}\f{d}{dt}\int_{\mathbb{R}^{2}} |\theta^{\varepsilon}|^{p}dx=(p-1)\int_{\mathbb{R}^{2}}(v_{j}\theta)^{\varepsilon} \partial_{j} \theta^{\varepsilon}| \theta^{\varepsilon}|^{p-2}dx.
$$
The incompressible condition allows us to formulate the above equation as
$$
\f{1}{p}\f{d}{dt}\int_{\mathbb{R}^{2}} |\theta^{\varepsilon}|^{p}dx=(p-1)\int_{\mathbb{R}^{2}}\big[(v_{j}\theta)^{\varepsilon} -v_{j}^{\varepsilon}\theta^{\varepsilon} \big]\partial_{j} \theta^{\varepsilon}| \theta^{\varepsilon}|^{p-2}dx,
$$
which immediately  means
$$
\frac{1}{p}\left(\|\theta^{\varepsilon}(x,t)\|_{L^{p} (\mathbb{R}^{2})}
-\|\theta^{\varepsilon}(x,0)\|_{L^{p} (\mathbb{R}^{2})}\right)= (p-1)\int^{t}_{0} \int_{\mathbb{R}^{2}}\big((v_{j}\theta)^{\varepsilon} -v_{j}^{\varepsilon}\theta^{\varepsilon} \big)\partial_{j} \theta^{\varepsilon} |\theta^{\varepsilon}|^{p-2}dxds.
$$
The H\"older inequality enables us to get
\begin{equation}\label{e4.3}
	\begin{aligned}
	&\B|\int^{t}_{0}\int_{\mathbb{R}^{2}}\big[(v_{j}\theta)^{\varepsilon} -v_{j}^{\varepsilon}\theta^{\varepsilon} \big]\partial_{i} \theta^{\varepsilon}| \theta^{\varepsilon}|^{p-2}dxds\B|\\
\leq	&C  \|(v_{j}\theta)^{\varepsilon} -v_{j}^{\varepsilon}\theta^{\varepsilon}\|_{L^{\frac{p+1}{2}}(0,T;L^{\f{p+1}{2}} (\mathbb{R}^{2}))}  \|\partial_{j}\theta^{\varepsilon}\|_{L^{p+1}(0,T;L^{p+1} (\mathbb{R}^{2}))} \||\theta^{\varepsilon}|^{p-2}\|_{L^{\frac{p+1}{p-2}}(0,T;L^{\f{p+1}{p-2}} (\mathbb{R}^{2}))}.
\end{aligned}\end{equation}
Since
$B^s_{p,q}=\dot{B}^s_{p,q}\cap L^{p}$ for $s>0$, the hypothesis $ \theta\in L^{p+1}(0,T;B^{\alpha}_{p+1,c(\mathbb{N})})$ means $ \theta\in L^{p+1}(0,T;\dot{B}^{\alpha}_{p+1,c(\mathbb{N})})$.
This  and the boundedness of Riesz transforms in homogeneous Besov spaces, we obtain
$$
v=\mathcal{R}^{\perp}\Lambda^{\gamma-1}\theta\in   L^{p+1}(0,T;\dot{B}^{\alpha-\gamma+1}_{p+1,c(\mathbb{N})}).
$$
Combining  $\theta\in L^{p+1}(0,T;\dot{B}^{\alpha}_{p+1,c(\mathbb{N})})$ with  $v \in   L^{p+1}(0,T;\dot{B}^{\alpha-\gamma+1}_{p+1,c(\mathbb{N})})$ and invoking  Lemma \ref{lem2.3}, we see that
\begin{equation}\label{e4.4}
\|(v_{j}\theta)^{\varepsilon} -v_{j}^{\varepsilon}\theta^{\varepsilon}\|_{L^{\f{p+1}{2}}(0,T;L^{\f{p+1}{2}} (\mathbb{R}^{2}))}\leq o(\varepsilon^{2\alpha-\gamma+1}),
\end{equation}
where we require $0<\alpha<1$ and  $0<\alpha-\gamma+1<1$.\\
 Using Lemma \ref{lem2.2}, we know that
\begin{equation}\label{e4.5}
\|\partial_{j}\theta^{\varepsilon}\|_{L^{p+1}(0,T;L^{p+1} (\mathbb{R}^{2}))}\leq o(\varepsilon^{\alpha-1}).
\end{equation}
Moreover,  in view of  the definition of Besov spaces, we have
\begin{equation}\label{e4.6} \||\theta^{\varepsilon}|^{p-2}\|_{L^{\f{p+1}{p-2}}(0,T;L^{\f{p+1}{p-2}} (\mathbb{R}^{2}))}\leq C \| \theta^{\varepsilon} \|^{p-2}_{L^{p+1}(0,T;L^{ p+1 } (\mathbb{R}^{2}))}\leq C \| \theta  \|^{p-2}_{L^{p+1}(0,T;B^{\alpha}_{p+1,c(\mathbb{N})})}.
\end{equation}
 Then substituting \eqref{e4.4}-\eqref{e4.6} into \eqref{e4.3}, setting $ \alpha=\frac{\gamma}{3}$ and choosing  $\varepsilon\rightarrow 0$ with $0<\gamma <\frac{3}{2}$, we have
 $$\B|\int^{t}_{0}\int\big[(v_{j}\theta)^{\varepsilon} -v_{j}^{\varepsilon}\theta^{\varepsilon} \big]\partial_{i} \theta^{\varepsilon}| \theta^{\varepsilon}|^{p-2}dxds\B|\leq o(\varepsilon^{3\alpha-\gamma}) \| \theta  \|^{p-2}_{L^{p+1}(0,T;B^{\alpha}_{p+1,c(\mathbb{N})})}\rightarrow 0.$$
Then  we have completed  the proof of the first part of Theorem \ref{the1.1}.
By a similar argument to \eqref{e4.4}-\eqref{e4.6}, we can conclude the second part of Theorem \ref{the1.1} for $ \theta\in L^{p+1}(0,T;{B}^{\alpha}_{p+1,\infty}).$
\end{proof}

\begin{proof}[Proof of Corollary \ref{coro1.1}]
It is enough to notice that
$$
\f{1}{2}\f{d}{dt}\int_{\mathbb{R}^{2}} |S_{N}\theta|^{2}dx= \int_{\mathbb{R}^{2}}\big[S_{N}(v_{j}\theta)  -S_{N}v_{j} S_{N}\theta \big]\partial_{j} S_{N}\theta dx.
$$
Exactly as in the above derivation in the Theorem \ref{the1.1},   the  proof of this Corollary.
\end{proof}
Next, we present the proof of Theorem \ref{the1.2}. In the proof of Theorem \ref{the1.1} ,  it suffices to replace \eqref{3.4} by
$$\ba&\| v_{j}(x-y)-v_{j}(x) \|_{ L^{p+1}  (\mathbb{R}^{2})} \\
\leq&C\left( 2^{N(1-\alpha)}|y|\sum_{j\leq N}2^{-(N-j)(1-\alpha)}2^{j\alpha}\|\dot{\Delta}_{j}v\|_{L^{p+1} (\mathbb{R}^{2})}+2^{-\alpha N}\sum_{j>N} 2^{(N-j)(\alpha)} 2^{j\alpha}\|\dot{\Delta}_{j}v\|_{L^{p+1} (\mathbb{R}^{2})}\right)\\
\leq& C\left[ 2^{N(1-\alpha)}|y| +2^{ -\alpha  N}\right]\left(K_{1}\ast \dot{d}_j\right)(N)\\
\leq & C( 2^{N}|y| +1)2^{  -\alpha  N}\left(K_{1}\ast \dot{d}_j\right)(N),
\ea$$
where
$$K_{1}(j)=\left\{\begin{aligned}
	&2^{j\alpha},~~~~~~~~\text{if}~j\leq0,\\
	& 2^{-(1-\alpha)j},~\text{if}~j>0,
\end{aligned}\right.
$$
and   $\dot{d}_{1}(j)=2^{j\alpha}\|\dot{\Delta}_{j}v\|_{L^{p+1}}$. We omit the details here. We only outline its proof by Constantin-E-Titi type commutator estimates  in  physical Onsager type spaces in the following.
\begin{proof}[Proof of Theorem \ref{the1.2}]
Based on the second proof of Theorem \ref{the1.1}, we just give the key estimates.
It follows from the H\"older inequality, we discover that
\begin{equation}\label{4.3}
	\begin{aligned}
	&\B|\int^{t}_{0}\int_{ \mathbb{R}^{2}}\big[(v_{j}\theta)^{\varepsilon} -v_{j}^{\varepsilon}\theta^{\varepsilon} \big]\partial_{i} \theta^{\varepsilon}| \theta^{\varepsilon}|^{p-2}dxds\B|\\
\leq	&C  \|(v_{j}\theta)^{\varepsilon} -v_{j}^{\varepsilon}\theta^{\varepsilon}\|_{L^{\frac{r_1 r_2}{r_1+r_2}}(0,T;L^{\f{p+1}{2}} (\mathbb{R}^{2}))}  \|\partial_{j}\theta^{\varepsilon}\|_{L^{r_2}(0,T;L^{p+1} (\mathbb{R}^{2}))} \||\theta^{\varepsilon}|^{p-2}\|_{L^{p_4}(0,T;L^{\f{p+1}{p-2}} (\mathbb{R}^{2}))},
\end{aligned}\end{equation}
where $\frac{r_1+r_2}{r_1 r_2}+\frac{1}{r_{2}}+\frac{1}{p_{4}}=1.$

From
$v \in L^{r_1}(0,T;\dot{B}^{\frac{1}{3}}_{p+1,c(\mathbb{N})})$ and $ \theta\in L^{r_2}(0,T;\dot{B}^{\frac{1}{3}}_{p+1,\infty})$, we deduce from Lemma \ref{lem2.3} that
\begin{equation}\label{4.4}
\|(v_{j}\theta)^{\varepsilon} -v_{j}^{\varepsilon}\theta^{\varepsilon}\|_{L^{\frac{r_1 r_2}{r_1+r_2}}(0,T;L^{\f{p+1}{2}}(\mathbb{R}^{2}))}\leq C\text{o}(\varepsilon^{\frac{2}{3}}).
\end{equation}
From Lemma \ref{lem2.2}, we infer that
\begin{equation}\label{4.5}
\|\partial_{j}\theta^{\varepsilon}\|_{L^{r_2}(0,T;L^{p+1} (\mathbb{R}^{2}))}\leq C\text{O}(\varepsilon^{-\frac{2}{3}}).
\end{equation}
According to the definition of Besov spaces, we have
\begin{equation}\label{4.6} \||\theta^{\varepsilon}|^{p-2}\|_{L^{p_{4}}(0,T;L^{\f{p+1}{p-2}} (\mathbb{R}^{2}))}\leq C \| \theta^{\varepsilon} \|^{p-2}_{L^{p_{4}(p-2)}(0,T;L^{ p+1 } (\mathbb{R}^{2}))} \leq C \| \theta  \|_{L^{r_{2}}(0,T;B^{\frac{1}{3}}_{p+1,\infty})}^{p-2},
\end{equation}
where we used $p_{4}(p-2)=r_{2}$, which means $\frac{p}{r_2}+\frac{1}{r_1}=1$ and $p\geq 2$.

 Then substituting \eqref{4.4}-\eqref{4.6} into \eqref{4.3} and letting $\varepsilon\rightarrow 0$, we have
 $$\B|\int^{t}_{0}\int_{\mathbb{R}^{2}}\big[(v_{j}\theta)^{\varepsilon} -v_{j}^{\varepsilon}\theta^{\varepsilon} \big]\partial_{i} \theta^{\varepsilon}| \theta^{\varepsilon}|^{p-2}dxds\B|\leq C\text{o}(1)\| \theta \|^{p-2}_{L^{p_{4}(p-2)}(0,T;B^{\frac{1}{3}}_{p+1,\infty})}\rightarrow 0.$$
Then we have completed the proof of Theorem \ref{the1.2}.
\end{proof}
\section{ General   helicity   conservation for 2-D surface quasi-geostrophic equations}
In this section, we are concerned with  the helicity conservation of weak solutions for 2-D generalized  surface quasi-geostrophic equation \eqref{gqg}. We also show two different approaches to prove Theorem \ref{the1.3}.
\begin{proof}[Proof of Theorem \ref{the1.3}]{\bf Approach 1: Littlewood-Paley theory}
First, due to the divergence free of velocity $v(x,t)$ and applying the operator $S_{N}$ to  the surface quasi-geostrophic equation \eqref{gqg}, we get
$$S_{N}\theta _{t}+S_{N}\partial_{j} (v_{j}\theta) =0,$$
and
$$\partial_{i}S_{N}\theta _{t}+\partial_{j}S_{N} (\partial_{i}v_{j}\theta) +\partial_{j} S_{N}(v_{j}\partial_{i}\theta) =0.$$
Straightforward calculations show that
\be\label{3.3}\ba
&\f{d}{dt}\int_{\mathbb{R}^2} S_{N}\theta \partial_{i}S_{N}\theta dx\\
=&\int_{\mathbb{R}^2} S_{N}\theta \partial_{i}\partial_{t}S_{N}\theta dx
+\int_{\mathbb{R}^2}\partial_{t}S_{N}\theta \partial_{i}S_{N}\theta dx\\
=&-\int_{\mathbb{R}^2} S_{N}\theta [\partial_{j} S_{N}(\partial_{i}v_{j}\theta)+\partial_{j} S_{N}(v_{j}\partial_{i}\theta) ]dx-\int_{\mathbb{R}^2} \partial_{j}S_{N} (v_{j}\theta) \partial_{i}S_{N}\theta dx,i=1,2.
\ea\ee
Thanks to $\int_{\mathbb{R}^2} \partial_{j} (\partial_{i}S_{N}v_{j}S_{N}\theta )S_{N}\theta dx=0$, we may write
$$\ba
&\f{d}{dt}\int_{\mathbb{R}^2} S_{N}\theta \partial_{i}S_{N}\theta dx\\
=&-\int_{\mathbb{R}^2} S_{N}\theta  \big[\partial_{j} S_{N}(\partial_{i}v_{j}\theta) -\partial_{j} (\partial_{i}S_{N}v_{j} S_{N}\theta \big] dx-\int_{\mathbb{R}^2} S_{N}\theta \big[\partial_{j} S_{N}(v_{j}\partial_{i}\theta)  -\partial_{j} (S_{N}v_{j} \partial_{i}S_{N}\theta )\big] dx\\&-\int_{\mathbb{R}^2} S_{N}\theta \partial_{j} (S_{N} v_{j} \partial_{i}S_{N}\theta ) dx  -\int_{\mathbb{R}^2}  \partial_{j}S_{N} (v_{j}\theta) \partial_{i}S_{N}\theta dx\\
=&\int_{\mathbb{R}^2}\partial_{j} S_{N}\theta \big [S_{N} (\partial_{i}v_{j}\theta) - (\partial_{i}S_{N}v_{j} S_{N}\theta )\big] dx+\int_{\mathbb{R}^2} \big[ S_{N}(v_{j}\partial_{i}\theta)  - (S_{N}v_{j} \partial_{i}S_{N}\theta)\big] \partial_{j}S_{N}\theta dx\\&-\int_{\mathbb{R}^2} S_{N}\theta S_N v_{j}\partial_{i}\partial_{j}S_{N}\theta dx  +\int_{\mathbb{R}^2} S_{N} (v_{j}\theta) \partial_{j}\partial_{i}S_{N}\theta dx\\
=&\int_{\mathbb{R}^2}\partial_{j} S_{N}\theta\big[S_{N} (\partial_{i}v_{j}\theta) - (\partial_{i}S_{N}v_{j} S_{N}\theta)\big] dx+\int_{\mathbb{R}^2}\big[ S_{N}(v_{j}\partial_{i}\theta)  - (S_{N}v_{j}\partial_{i}S_{N}\theta)\big] \partial_{j}S_{N}\theta dx\\&+\int_{\mathbb{R}^2}\big[S_{N}(v_{j}\theta)-S_{N}\theta S_{N}v_{j}\big]\partial_{i}\partial_{j}S_{N}\theta  dx\\
=&I+II+III.
\ea$$
To control $I$, we deduce from  the H\"older inequality that
$$
|I|\leq\|S_{N} (\partial_{i}v_{j}\theta) - (\partial_{i}S_{N}v_{j} S_{N}\theta) \|_{L^{\f65}} \|\partial_{j} S_{N}\theta\|_{L^{6}}.
$$
Taking advantage of Constantin-E-Titi identity
\eqref{ceti}, Minkowski inequality and the Sobolev inequality, we infer that
\begin{align}
&\|S_{N} (\partial_{i}v_{j}\theta) - (\partial_{i}S_{N}v_{j} S_{N}\theta) \|_{L^{\f65} (\mathbb{R}^{2})}\nonumber\\\leq& C2^{ 2N}\int_{\mathbb{R}^{2}}|\tilde{h}(2^{N}y)|\|\partial_{i} v_{j}(x-y)-\partial_{i}v_{j}(x) \|_{L^{\f32} (\mathbb{R}^{2}) }
 \| \theta(x-y)-\theta(x) \|_{L^{6}  (\mathbb{R}^{2})}dy\nonumber
\\&+C\|\partial_{i} v_{j}-S_{N}\partial_{i}v_{j} \|_{L^{\f32}  (\mathbb{R}^{2})} \| \theta-S_{N}\theta \|_{ L^{6} (\mathbb{R}^{2})}\nonumber \\\leq& C2^{ 2N}\int_{\mathbb{R}^{2}}|\tilde{h}(2^{N}y)|\|\partial_{i} v_{j}(x-y)-\partial_{i}v_{j}(x) \|_{L^{\f32} (\mathbb{R}^{2}) }
 \| \nabla\theta(x-y)-\nabla\theta(x) \|_{L^{\f32}  (\mathbb{R}^{2})}dy
\nonumber\\&+C\|\partial_{i} v_{j}-S_{N}\partial_{i}v_{j} \|_{L^{\f32}  (\mathbb{R}^{2})} \| \nabla\theta-\nabla S_{N}\theta \|_{L^{\f32} (\mathbb{R}^{2})}.\label{14.3}
\end{align}
Arguing in the same manner as \eqref{3.4},
we observe that
\begin{align}
&\|\partial_{i} v_{j}(x-y)-\partial_{i}v_{j}(x) \|_{L^{\f32} (\mathbb{R}^{2}) }\nonumber\\
\leq
  &C\B(\sum_{j\leq N}2^{j}|y|\|\dot{\Delta}_{j}\nabla v\|_{L^{\f32} (\mathbb{R}^{2})}+\sum_{j>N}\|\dot{\Delta}_{j}\nabla v\|_{L^{\f32} }\B)\nonumber \\
\leq&C\B( 2^{N(\gamma-\alpha)}|y|\sum_{j\leq N}2^{-(N-j)(\gamma-\alpha)}2^{j\alpha}
\|\dot{\Delta}_{j}\nabla\theta\|_{L^{\f32} (\mathbb{R}^{2})}\nonumber\\&+2^{(\gamma-1-\alpha )N}\sum_{j>N} 2^{(N-j)(\alpha+1-\gamma)} 2^{j\alpha}\|\dot{\Delta}_{j}\nabla\theta\|_{L^{\f32} (\mathbb{R}^{2}) }\B)\nonumber\\
\leq& C\left[ 2^{N(\gamma-\alpha)}|y| +2^{(\gamma-1-\alpha )N}\right]\left(K_{1}\ast \dot{\tilde{d}}_{j}\right)(N)\nonumber\\
\leq & C( 2^{N}|y| +1)2^{(\gamma-1-\alpha )N}\left(K_{1}\ast \dot{\tilde{d}}_{j}\right)(N).\label{1e4.4}
\end{align}
where $ K_{1}$ is defined in \eqref{K1} and
$\dot{\tilde{d}}_{j}=2^{j\alpha}\|\dot{\Delta}_{j}\nabla\theta\|_{L^{\frac{3}{2}}}.$

Similar to the   derivation of  \eqref{3.5} and using \eqref{K2}, we get
\be\ba\label{h4.5}
&\| \nabla\theta(x-y)-\nabla\theta(x) \|_{L^{\f32} (\mathbb{R}^{2})}\\
\leq &C\B(\sum_{j\leq N}2^{j}|y|\|\dot{\Delta}_{j}\nabla\theta\|_{L^{\f32}(\mathbb{R}^{2})}
+\sum_{j>N}\|\dot{\Delta}_{j}\nabla\theta\|_{L^{\f32}(\mathbb{R}^{2})}\B)\\
\leq&C\B( 2^{N(1-\alpha)}|y|\sum_{j\leq N}2^{-(N-j)(1-\alpha)}2^{j\alpha}\|\dot{\Delta}_{j}\nabla\theta\|_{L^{\f32} (\mathbb{R}^{2})}+2^{-\alpha N}\sum_{j>N} 2^{(N-j)\alpha} 2^{j\alpha}\|\dot{\Delta}_{j}\nabla\theta\|_{L^{\f32} (\mathbb{R}^{2})}\B)\\
\leq& C\B[ 2^{N(1-\alpha)}|y| +2^{ -\alpha N}\B]\left(K_{2}\ast \dot{\tilde{d}}\right)(N)\\
\leq& C(2^{N}|y| +1)2^{ -\alpha N}\left(K_{2}\ast \dot{\tilde{d}}_{j}\right)(N).
\ea\ee
  Some straightforward computations yields
\be\label{h4.6}\| \nabla\theta-\nabla S_{N}\theta \|_{L^{\f32}(\mathbb{R}^{2})} \leq C 2^{ -\alpha N}\left(K_{2}\ast \dot{\tilde{d}}_{j}\right)(N),
\ee
and
\be\label{h4.7}\ba \| \partial_{j} S_{N}\theta\|_{L^{\frac{3}{2}}(\mathbb{R}^{2})}&\leq C\|\nabla\partial_{j} S_{N}\theta\|_{L^{\frac{3}{2}}(\mathbb{R}^{2})}\leq  C \sum_{j\leq N}2^{j} \| \Delta_{j}\nabla\theta\|_{L^{\frac{3}{2}}(\mathbb{R}^{2})}\leq  2^{N(1-\alpha)}\left(K_{2}\ast \tilde{d}_{j}\right)(N),
\ea\ee
where the Sobolev embedding was used and
$\tilde{d}_{j}=2^{j\alpha}\|\Delta_{j}\nabla\theta\|_{L^{\frac{3}{2}} (\mathbb{R}^{2})}.$\\
As a consequence, we know
$$I\leq 2^{(\gamma-1-\alpha )N}\left(K_{1}\ast \dot{\tilde{d}}_{j}\right)(N)2^{ -\alpha N}\left(K_{2}\ast \dot{\tilde{d}}_{j}\right)(N) 2^{N(1-\alpha)}\left(K_{2}\ast \tilde{d}_{2}\right)(N).
$$
Repeating the deduction process of $I$, we have
$$II\leq 2^{(\gamma -3\alpha )N}\left(K_{1}\ast \dot{\tilde{d}}_{2}\right)(N) \left(K_{2}\ast \dot{\tilde{d}}_{2}\right)(N) \left(K_{2}\ast \tilde{d}_{2}\right)(N).
$$
Taking advantage of  the H\"older inequality, we infer  that
\be\label{h4.8}
III\leq \|S_{N}(v_{j}\theta)-S_{N}\theta S_{N}v_{j}\|_{L^{3} (\mathbb{R}^{2})}  \|\partial_{i}\partial_{j}S_{N}\theta \|_{L^{\f32} (\mathbb{R}^{2}) }.
\ee
Following the path of \eqref{1e4.4}, we arrive at
$$\ba
 &\|S_{N}(v_{j}\theta)-S_{N}\theta S_{N}v_{j}\|_{L^{3} (\mathbb{R}^{2})}\\
 \leq& C2^{ 2N}\int|\tilde{h}(2^{N}y)|  \|v_{j}(x-y)- v_{j}(x) \|_{L^{6}  (\mathbb{R}^{2})}
 \| \theta(x-y)-\theta(x) \|_{L^{6}  (\mathbb{R}^{2})}dy
\\&+C\|  v_{j}-S_{N} v_{j} \|_{L^{6}  (\mathbb{R}^{2})} \| \theta-S_{N}\theta \|_{ L^{6} (\mathbb{R}^{2})}\\
 \leq& C2^{ 2N}\int|\tilde{h}(2^{N}y)|  \|\nabla v_{j}(x-y)- \nabla v_{j}(x) \|_{L^{\f32}  (\mathbb{R}^{2})}
 \| \nabla\theta(x-y)-\nabla\theta(x) \|_{L^{\f32} (\mathbb{R}^{2}) }dy
\\&+C\|  \nabla v_{j}-S_{N} \nabla v_{j} \|_{L^{\f32} (\mathbb{R}^{2}) } \| \nabla\theta-\nabla S_{N}\theta \|_{ L^{\f32} (\mathbb{R}^{2})}.
 \ea$$
From \eqref{1e4.4}-\eqref{h4.5}, we have
\be\label{h4.9}
\|S_{N}(v_{j}\theta)-S_{N}\theta S_{N}v_{j}\|_{L^{3} (\mathbb{R}^{2})}
 \leq 2^{(\gamma-1-\alpha )N}\left(K_{1}\ast \dot{\tilde{d}}_{j}\right)(N)2^{ -\alpha N}\left(K_{2}\ast \dot{\tilde{d}}_{j}\right)(N). \ee
It follows from \eqref{h4.7} that
\be\label{h4.10}
\|\nabla\partial_{j} S_{N}\theta\|_{L^{\f32} (\mathbb{R}^{2})}\leq  C \sum_{j\leq N}2^{j} \| \Delta_{j}\nabla\theta\|_{L^{\f32}(\mathbb{R}^{2})}\leq  2^{N(1-\alpha)}\left(K_{2}\ast \tilde{d}_{j}\right)(N).
\ee
Substituting \eqref{h4.9}  \eqref{h4.10} into  \eqref{h4.8}, we conclude that
$$III\leq C 2^{(\gamma -3\alpha )N}\left(K_{1}\ast \dot{\tilde{d}}_{j}\right)(N) \left(K_{2}\ast \dot{\tilde{d}}_{j}\right)(N) \left(K_{2}\ast \tilde{d}_{j}\right)(N).
$$
Finally, we end up with
\begin{equation}
	\begin{aligned}
		\f{d}{dt}\int S_{N}\theta \partial_{i}S_{N}\theta dx&\leq C 2^{(\gamma -3\alpha )N}\left(K_{1}\ast \dot{\tilde{d}}_{j}\right)(N)\left( K_{2}\ast \dot{\tilde{d}}_{j}\right)(N) \left(K_{2}\ast \tilde{d}_{j}\right)(N).
	\end{aligned}
\end{equation}
At this stage, the rest proof of this theorem  is the same as the  one  of Theorem
\ref{the1.1}.
\end{proof}
\begin{proof}[Proof of Theorem \ref{the1.3}]{\bf Approach 2: Constantin-E-Titi type commutator estimates  in  physical Onsager type spaces}
  It is obvious that
$$\theta^{\varepsilon}_{t}+\partial_{j} (v_{j}\theta)^{\varepsilon}=0,$$
and
$$\partial_{i}\theta^{\varepsilon}_{t}+\partial_{j} (\partial_{i}v_{j}\theta)^{\varepsilon}+\partial_{j} (v_{j}\partial_{i}\theta)^{\varepsilon}=0.$$
Thus, it follows from the direct computation that
\be\label{h4.3}\ba
\f{d}{dt}\int_{\mathbb{R}^2}\theta^{\varepsilon}\partial_{i}\theta^{\varepsilon}dx
&=\int_{\mathbb{R}^2}\theta^{\varepsilon}\partial_{i}\partial_{t}\theta^{\varepsilon}dx
+\int_{\mathbb{R}^2}\partial_{t}\theta^{\varepsilon}\partial_{i}\theta^{\varepsilon}dx\\
&=-\int_{\mathbb{R}^2}\theta^{\varepsilon}[\partial_{j} (\partial_{i}v_{j}\theta)^{\varepsilon}+\partial_{j} (v_{j}\partial_{i}\theta)^{\varepsilon}]dx-\int_{\mathbb{R}^2} \partial_{j} (v_{j}\theta)^{\varepsilon}\partial_{i}\theta^{\varepsilon}dx,i=1,2.
\ea\ee
Since $\int_{\mathbb{R}^2} \partial_{j} (\partial_{i}v_{j}^{\varepsilon}\theta^{\varepsilon})\theta^{\varepsilon}dx=0$, we can rewrite the above equation \eqref{h4.3} as
$$\ba
\f{d}{dt}\int_{\mathbb{R}^{2}}\theta^{\varepsilon}\partial_{i}\theta^{\varepsilon}dx
=&-\int_{\mathbb{R}^{2}}\theta^{\varepsilon} \big[\partial_{j} (\partial_{i}v_{j}\theta)^{\varepsilon}-\partial_{j} (\partial_{i}v_{j}^{\varepsilon}\theta^{\varepsilon})\big] dx-\int_{\mathbb{R}^2}\theta^{\varepsilon}\big[\partial_{j} (v_{j}\partial_{i}\theta)^{\varepsilon} -\partial_{j} (v_{j}^{\varepsilon}\partial_{i}\theta^{\varepsilon})\big] dx\\&-\int_{\mathbb{R}^{2}}\theta^{\varepsilon}\partial_{j} (v_{j}^{\varepsilon}\partial_{i}\theta^{\varepsilon}) dx  -\int_{\mathbb{R}^{2}} \partial_{j} (v_{j}\theta)^{\varepsilon}\partial_{i}\theta^{\varepsilon}dx\\
=&\int_{\mathbb{R}^{2}}\partial_{j} \theta^{\varepsilon}\big [ (\partial_{i}v_{j}\theta)^{\varepsilon}- (\partial_{i}v_{j}^{\varepsilon}\theta^{\varepsilon})\big] dx+\int_{\mathbb{R}^{2}}\big[ (v_{j}\partial_{i}\theta)^{\varepsilon} - (v_{j}^{\varepsilon}\partial_{i}\theta^{\varepsilon})\big] \partial_{j}\theta^{\varepsilon} dx\\&-\int_{\mathbb{R}^{2}}\theta^{\varepsilon} v_{j}^{\varepsilon}\partial_{i}\partial_{j}\theta^{\varepsilon}  dx  +\int_{\mathbb{R}^{2}} (v_{j}\theta)^{\varepsilon}\partial_{j}\partial_{i}\theta^{\varepsilon}dx\\
=&\int_{\mathbb{R}^{2}}\partial_{j} \theta^{\varepsilon} \big[ (\partial_{i}v_{j}\theta)^{\varepsilon}- (\partial_{i}v_{j}^{\varepsilon}\theta^{\varepsilon})\big] dx+\int_{\mathbb{R}^2}\big[ (v_{j}\partial_{i}\theta)^{\varepsilon} - (v_{j}^{\varepsilon}\partial_{i}\theta^{\varepsilon})\big] \partial_{j}\theta^{\varepsilon} dx\\&+\int_{\mathbb{R}^{2}}\big[(v_{j}\theta)^{\varepsilon}-\theta^{\varepsilon} v_{j}^{\varepsilon}\big]\partial_{i}\partial_{j}\theta^{\varepsilon}  dx,
\ea$$
which implies
$$\ba
&\int_{\mathbb{R}^{2}}\theta^{\varepsilon}(x,t)\partial_{i}\theta^{\varepsilon}(x,t)dx
-\int_{\mathbb{R}^{2}}\theta^{\varepsilon}(x,0)\partial_{i}\theta^{\varepsilon}(x,0)dx
\\
= &\int_{0}^{t}\int_{\mathbb{R}^{2}}\partial_{j} \theta^{\varepsilon}\big [ (\partial_{i}v_{j}\theta)^{\varepsilon}- (\partial_{i}v_{j}^{\varepsilon}\theta^{\varepsilon})\big] dxds+\int_{0}^{t}\int_{\mathbb{R}^{2}}\big[ (v_{j}\partial_{i}\theta)^{\varepsilon} - (v_{j}^{\varepsilon}\partial_{i}\theta^{\varepsilon})\big] \partial_{j}\theta^{\varepsilon} dxds\\&+\int_{0}^{t}\int_{\mathbb{R}^{2}}\big[(v_{j}\theta)^{\varepsilon}-\theta^{\varepsilon} v_{j}^{\varepsilon}\big]\partial_{i}\partial_{j}\theta^{\varepsilon}  dxds  \\
=&
I+II+III.
\ea$$
Taking advantage of  the H\"older inequality, we get
\be\label{h4.11}
|I|\leq\|(\partial_{i}v_{j}\theta)^{\varepsilon}- (\partial_{i}v_{j}^{\varepsilon}\theta^{\varepsilon}) \|_{L^{\f32}(0,T;L^{\f65}(\mathbb{R}^{2}))} \| \partial_{j} \theta^{\varepsilon}\|_{L^{3}(0,T;L^{6}(\mathbb{R}^{2}))}.
\ee
Due to the hypothesis $ \nabla\theta\in L^{3}(0,T;\dot{B}^{\alpha}_{\frac{3}{2},c(\mathbb{N})})$ and the boundedness of Riesz transforms in homogeneous Besov spaces, we obtain
that
$$\nabla v=\mathcal{R}^{\perp}\Lambda^{\gamma-1}\nabla\theta\in   L^{3}(0,T;\dot{B}^{\alpha-\gamma+1}_{\frac{3}{2},c(\mathbb{N})}).$$
Then we employ (3) in Lemma \ref{lem2.3}
with
$q=\f65,d=2, q_{1}=\f32, q_{2}=\f32<2 $ and use
$\nabla\theta\in L^{3}(0,T;\dot{B}^{\alpha}_{ \f32, c(\mathbb{N})})$ and  $\nabla v\in L^{3}(0,T;\dot{B}^{\alpha-\gamma+1}_{ \f32, c(\mathbb{N})})$ to derive that
\be\label{h4.12}
\|(\partial_{i}v_{j}\theta)^{\varepsilon}- (\partial_{i}v_{j}^{\varepsilon}\theta^{\varepsilon}) \|_{L^{\f32}(0,T;L^{\f65} (\mathbb{R}^{2}))}
\leq \text{o}(\varepsilon^{2\alpha-\gamma+1}),\ee
where $ \f12+\f56=\frac{1}{q_{1}}+\frac{1}{q_{2}}$  and  $0<\alpha-\gamma+1<1$ were utilized.
\\
By means of
Sobolev embedding theorem, Lemma \ref{lem2.2} and $\nabla\theta\in L^{3}(0,T;\dot{B}^{\alpha}_{ \f32, c(\mathbb{N})}) $, we conclude that
\be\label{h4.13}
\| \partial_{j} \theta^{\varepsilon}\|_{L^{3}(0,T;L^{6} (\mathbb{R}^{2}))}\leq C\|\nabla \partial_{j} \theta^{\varepsilon}\|_{L^{3}(0,T;L^{\f32} (\mathbb{R}^{2}))}\leq C \text{o}(\varepsilon^{\alpha-1} ).
\ee
Plugging \eqref{h4.12} and \eqref{h4.13} into \eqref{h4.11}, we arrive at
$$
|I|\leq C \text{o}(\varepsilon^{3\alpha-\gamma}).
$$
Arguing as above, we deduce that
$$
|II|\leq C \text{o}(\varepsilon^{3\alpha-\gamma}).
$$
It is enough  to estimate the term $III$. Applying the H\"older inequality once again, we get
\be\label{h4.14}
|III|\leq  \|(v_{j}\theta)^{\varepsilon}-\theta^{\varepsilon} v_{j}^{\varepsilon} \|_{L^{\f32}(0,T;L^{3} (\mathbb{R}^{2}))} \| \partial_{i}\partial_{j}\theta^{\varepsilon}\|_{L^{3}(0,T;L^{\f32} (\mathbb{R}^{2}))}.
\ee
Invoking (2) in Lemma \ref{lem2.3}
with
$q=3,d=2, q_{1}=\f32=q_{2}=\f32<2,  $  we find
\be\label{h4.15}
 \|(v_{j}\theta)^{\varepsilon}-\theta^{\varepsilon} v_{j}^{\varepsilon} \|_{L^{\f32}(0,T;L^{3} (\mathbb{R}^{2}))}\leq C \text{o}(\varepsilon^{2\alpha-\gamma+1}),
 \ee
where we used $\nabla\theta\in L^{3}(0,T;\dot{B}^{\alpha}_{ \f32, c(\mathbb{N})})$ and $\nabla v\in L^{3}(0,T;\dot{B}^{\alpha-\gamma+1}_{ \f32, c(\mathbb{N})})$.

In the light of  Lemma \ref{lem2.2}, we infer that
\be\label{h4.16}
\| \partial_{i}\partial_{j}\theta^{\varepsilon}\|_{L^{3}(0,T;L^{\f32} (\mathbb{R}^{2}))}\leq C \text{o}(\varepsilon^{\alpha-1} ).
\ee
Substituting
\eqref{h4.15} and \eqref{h4.16} into \eqref{h4.14},  we see that
$$|III|\leq C \text{o}(\varepsilon^{3\alpha-\gamma}).$$
Since we need $0<\alpha-\gamma+1<1$, we discuss in two cases   $0<\gamma <\frac{3}{2}$  and  $\frac{3}{2}\leq\gamma <2$ as Theorem \ref{the1.1}.   This enables us to complete the proof.
\end{proof}

\section{Conclusion}
We apply the Littlewood-Paley theory as \cite{[CCFS]} and the
Constantin-E-Titi type commutator estimates  in  physical Onsager type spaces
 to study the   energy (helicity) conservation of weak solutions for  the 2-D generalized    quasi-geostrophic equation with the velocity $v$ determined by $v=\mathcal{R}^{\perp}\Lambda^{\gamma-1}\theta$ with $0<\gamma<2$, respectively. For
 the case $0<\gamma<\f32$, the   sufficient conditions for the energy (helicity) conservation of weak solutions of this equation   in Onsager's critical space  are derived. For the more singular case $ \f32\leq\gamma<2$, we obtain the corresponding results in critical spaces.
 Since the Littlewood-Paley decomposition and Besov space     and Lemma \ref{lem2.2} and \ref{lem2.3} are known for periodic domain,   the main results are also valid for periodic case.

 A natural question is to extend our results to other models which modifies the velocity.  A  possible candidate is the inviscid Leary-$\alpha$ or Euler-$\alpha$  system. After we completed the main part of this paper, we learned the energy conservation of these models recently   studied by Boutros-Titi in \cite{[BT]}   and         Beekie-Novack in \cite{[BN]}. Compared with their results, the results here give how the critical regularity for the energy   conservation of the weak solutions depends on the
 the parameter $\alpha$ of the velocity.

The
non-uniqueness of weak solutions to the standard  surface quasi-geostrophic equation \eqref{qg1} can be found in \cite{[BSV],[IM]}.
It would be interesting to  show
 the weak solutions to the generalized    quasi-geostrophic  equation \eqref{gqg}
   are not unique.

\section*{Acknowledgement}

 Wang was partially supported by  the National Natural
 Science Foundation of China under grant (No. 11971446, No. 12071113   and  No.  11601492). Ye was partially supported by the National Natural Science Foundation of China  under grant (No.11701145) and China Postdoctoral Science Foundation (No. 2020M672196). Yu was partially supported by the
National Natural Science Foundation of China (NNSFC) (No. 11901040), Beijing Natural
Science Foundation (BNSF) (No. 1204030) and Beijing Municipal Education Commission
(KM202011232020).

\end{document}